\documentclass{article}
\usepackage{amsthm, amsmath, amssymb, amscd, color}% mathrsfs,  enumerate, verbatim}
\usepackage[hyperfootnotes=false]{hyperref}

\theoremstyle{theorem}
\newtheorem{theorem}{Theorem}
\newtheorem{proposition}[theorem]{Proposition}
\newtheorem{lemma}[theorem]{Lemma}

\theoremstyle{definition}

\begin{document}
%MSC Subject Classification: 20D15 Finite nilpotent groups
\title{A Note on Finite Nilpotent Groups}
\markright{Nilpotent Groups}
\author{Nicholas J. Werner}

\maketitle

\begin{abstract}
\noindent It is well known that if $G$ is a group and $H$ is a normal subgroup of $G$ of finite index $k$, then $x^k \in H$ for every $x \in G$. We examine finite groups $G$ with the property that $x^k \in H$ for every subgroup $H$ of $G$, where $k$ is the index of $H$ in $G$. We prove that a finite group $G$ satisfies this property if and only if $G$ is nilpotent.
\end{abstract}

{\color{red}
\noindent \textbf{Author's Note added July 11, 2024}. I have learned that the results of this paper have already appeared in:\\
\indent L. Sabatini. Products of subgroups, subnormality, and relative orders\\
\indent of elements. Ars Math. Contemp. 24,  no. 1 (2024), 1--9. \href{https://doi.org/10.26493/1855-3974.2975.1b2}{[DOI]}\\
}

\noindent For a group $G$ and a subgroup $H \leq G$, let $[G:H]$ denote the index of $H$ in $G$. A common exercise in a first course on abstract algebra is to prove that if $H$ is normal in $G$ and $[G:H]$ is finite, then $x^{[G:H]} \in H$ for all $x \in G$. This result need not hold when $H$ is not normal. For example, take $G$ to be the symmetric group $S_3$, let $\sigma$ and $\tau$ be distinct transpositions in $G$, and let $H = \langle \sigma \rangle$. Then, $[G:H] = 3$, but $\tau^3 = \tau \notin H$. 

One could then ask: for which groups do we always have $x^{[G:H]} \in H$? We will restrict our focus to finite groups, and we will say that a finite group $G$ has property $(\ast)$ if the following holds:
\begin{equation}\label{star}\tag{$\ast$}
\text{ for each } H \leq G \text{ and for all } x \in G, x^{[G:H]} \in H.
\end{equation}
Certainly, \eqref{star} is true for finite Abelian groups, and more generally for any finite group with all subgroups normal (such as the quaternion group $Q_8$). However, the requirement that every subgroup of $G$ be normal in $G$ is not necessary for \eqref{star} to hold. For instance, the dihedral group of order 8 contains subgroups that are not normal, and it is straightforward to verify that this group satisfies \eqref{star}.

In this note, we prove that a finite group satisfies \eqref{star} if and only if it is nilpotent. Typically, nilpotent groups are defined using upper central series. For this, one constructs the chain of normal subgroups
\begin{equation*}
Z_0 \trianglelefteq Z_1 \trianglelefteq Z_2 \trianglelefteq \cdots
\end{equation*}
where $Z_0$ is the trivial subgroup, $Z_1$ is the center of $G$, and $Z_{i+1}$ is chosen so that $Z_{i+1}/Z_i$ equals the center of $G/Z_i$. The group $G$ is said to be \textit{nilpotent} if the above chain terminates at $G$ in a finite number of steps. When $G$ is finite, there are a variety of other ways to characterize nilpotent groups. The proposition below summarizes a number of these conditions; a proof can be found in many algebra texts such as \cite[Section 6.1]{DF}.

\begin{proposition}\label{nilpotent conditions} Let $G$ be a finite group. The following are equivalent.
\begin{enumerate}
\item[(1)] $G$ is nilpotent.
\item[(2)] Every proper subgroup of $G$ is a proper subgroup of its normalizer in $G$.
\item[(3)] Every Sylow subgroup of $G$ is normal.
\item[(4)] $G$ is isomorphic to the direct product of all of its Sylow subgroups.
\item[(5)] Every maximal subgroup of $G$ is normal.
%\item $G$ has a normal subgroup of order $d$ for each divisor $d$ of $|G|$.
\end{enumerate}
\end{proposition}

Additional characterizations of finite nilpotent groups have been given in \cite{Cocke, Holmes, Wong}. To prove that \eqref{star} can be added to this list of equivalent conditions, we begin by showing that \eqref{star} holds for any finite $p$-group. For this, we require two facts about finite $p$-groups. First, every maximal subgroup of a finite $p$-group has index $p$. Second, every maximal subgroup of a $p$-group is normal. We refer to \cite[Section 6.1]{DF} for proofs of these properties.

\begin{lemma}\label{p-group lemma}
Let $p$ be a prime and let $G$ be a group of order $p^n$, where $n \geq 1$. Then, $G$ satisfies \eqref{star}.
\end{lemma}
\begin{proof}
We use induction on $n$. If $n=1$, then $G$ is cyclic of order $p$ and the result is clear. So, assume that $n > 1$ and that \eqref{star} holds for all groups of order $p^{n-1}$.

Let $H$ be a proper subgroup of $G$. Then, there exists a maximal subgroup $M \leq G$ such that $H \leq M$. Let $[G:H] = p^k$, where $k \geq 1$. Since $M$ is maximal in $G$, the index of $M$ in $G$ equals $p$. Consequently, $[M:H] = p^{k-1}$. 

Since $M \trianglelefteq G$, $x^p \in M$ for all $x \in G$. By induction, \eqref{star} holds for $M$. Thus, for every $x \in G$, $x^{p^k} = (x^p)^{p^{k-1}} \in H$.
\end{proof}

\begin{lemma}\label{Subgroup lemma}
Let $G = G_1 \times \cdots \times G_n$, where $G_i$ is a finite group for each $1 \leq i \leq n$ and $\gcd(|G_i|, |G_j|) = 1$ whenever $i \ne j$. Then, every subgroup of $G$ has the form $H_1 \times \cdots \times H_n$, where $H_i \leq G_i$ for all $1 \leq i \leq n$.
\end{lemma}
\begin{proof}
This is a standard result, but we include a proof for the sake of completeness. It suffices to prove the lemma for the case where $G = G_1 \times G_2$. The general result then follows by induction.

Let $n_1 = |G_1|$ and $n_2 = |G_2|$. Assume that $G_1$ and $G_2$ have identity elements $e_1$ and $e_2$, respectively. Let $H \leq G$ and let $(x,y) \in H$. We first show that both $(x,e_2)$ and $(e_1,y)$ are in $H$. Since $\gcd(n_1,n_2)=1$, we see that $n_2$ is coprime to $|x|$ and $n_1$ is coprime to $|y|$. So, $x^{n_2}$ generates $\langle x \rangle$, and $y^{n_1}$ generates $\langle y \rangle$. Let $k_1, k_2$ be integers such that $x = x^{n_2k_2}$ and $y = y^{n_1 k_1}$. Then,
\begin{equation*}
(x,y)^{n_2k_2} = (x^{n_2k_2},e_2) = (x,e_2),
\end{equation*}
and likewise $(x,y)^{n_1k_1} = (e_1,y)$. Thus, $(x,e_2), (e_1,y) \in H$.

Now, let $H_1 = \{(h_1, h_2) \in H \mid h_2 = e_2\}$ and $H_2 = \{(h_1, h_2) \in H \mid h_1 = e_1\}$. Both $H_1$ and $H_2$ are normal subgroups of $H$, and their intersection is trivial. Moreover, by the previous paragraph, $H = H_1 H_2$. So, $H = H_1H_2 = H_1 \times H_2$, as desired.
\end{proof}

A subgroup $H$ of $G$ such that $|H|$ and $[G:H]$ are coprime is called a \textit{Hall subgroup} of $G$. For a Hall subgroup $H$, the condition that $x^{[G:H]} \in H$ for all $x \in G$ implies that $H \trianglelefteq G$.

\begin{lemma}\label{Hall subgroup lemma}
Let $G$ be a finite group and let $H \leq G$ be such that $|H|$ and $[G:H]$ are coprime. If $x^{[G:H]} \in H$ for all $x \in G$, then $H$ is normal in $G$.
\end{lemma}
\begin{proof}
Let $k=[G:H]$ and assume that $x^k \in H$ for all $x \in G$. Let $h \in H$. Since $k$ is coprime to $|H|$, it is also coprime to $|h|$. Thus, $\langle h^k \rangle = \langle h \rangle$, and so $h =h^{km}$ for some integer $m$. For each $g \in G$, we have
\begin{equation*}
ghg^{-1} = gh^{km}g^{-1} = (gh^mg^{-1})^k \in H,
\end{equation*}
so $H$ is normal in $G$.
\end{proof}

We can now prove that \eqref{star} characterizes finite nilpotent groups. In fact, it is enough that the condition in \eqref{star} holds only for Sylow subgroups of $G$.

\begin{theorem} Let $G$ be a finite group. The following are equivalent.
\begin{enumerate}
\item[(1)] $G$ is nilpotent.
\item[(2)] $G$ satisfies \eqref{star}.
\item[(3)] For each Sylow subgroup $P$ of $G$ and for all $x \in G$, $x^{[G:P]} \in P$.
\end{enumerate}
\end{theorem}
\begin{proof}
(1) $\Rightarrow$ (2) Assume that $G$ is nilpotent and let $p_1, \ldots, p_n$ be all of the distinct prime divisors of $|G|$. Since $G$ is nilpotent, $G$ has a unique Sylow $p_i$-subgroup $P_i$ for each $1 \leq i \leq n$, and we can assume (up to isomorphism) that $G = P_1 \times \cdots \times P_n$. Let $H \leq G$. By Lemma \ref{Subgroup lemma}, $H = H_1 \times \cdots \times H_n$, where $H_i \leq P_i$ for all $1 \leq i \leq n$. Let $k = [G:H]$ and let $x = (g_1, \ldots, g_n) \in G$. For each $i$, $[P_i:H_i]$ divides $k$, and \eqref{star} holds for $P_i$ by Lemma \ref{p-group lemma}, so $g_i^k \in H_i$. Thus, $x^k = (g_1^k, \ldots, g_n^k) \in H$.

(2) $\Rightarrow$ (3) This is clear.

(3) $\Rightarrow$ (1) Assume that $x^{[G:P]} \in P$ whenever $x \in G$ and $P$ is a Sylow subgroup of $G$. If $G$ is a $p$-group, then $G$ is nilpotent and we are done. If $G$ is not a $p$-group, then we can apply Lemma \ref{Hall subgroup lemma} to conclude that each Sylow subgroup of $G$ is normal in $G$. Hence, $G$ is nilpotent in this case as well.
\end{proof}

\end{document}